\documentclass[12pt]{amsart}
\usepackage{amsmath,amsthm,amsfonts,amssymb,eucal}
\renewcommand {\a}{ \alpha }

\newcommand{\vare}{\varepsilon}
\newcommand{\g}{\gamma}

\newcommand{\varf}{\varphi}
\renewcommand{\d}{\delta}

\newcommand{\D}{\Delta}
\newcommand{\s}{\sigma}
\newcommand{\Sg}{\Sigma}
\renewcommand{\l}{\lambda}

\newcommand{\z}{\zeta}

\newcommand{\vart}{\vartheta}

\newcommand{\om}{\omega}

\newcommand{\R}{ \mathbb R}

\newcommand{\N}{ \mathbb N}

\newcommand{\Z}{ \mathbb Z}

\newcommand{\CF}{\mathcal F}
\newcommand{\CG}{\mathcal G}
\newcommand{\CH}{\mathcal H}

\newcommand {\GH}{\mathfrak H}

\newcommand {\GS}{\mathfrak S}

\newcommand {\bh}{\mathbf h}

\newcommand {\bx}{\mathbf x}
\newcommand {\by}{\mathbf y}
\newcommand {\bb}{\mathbf b}

\newcommand {\BA}{\mathbf A}
\newcommand {\BB}{\mathbf B}

\newcommand {\BH}{\mathbf H}

\newcommand {\BT}{\mathbf T}

\newcommand{\gz}{\mathfrak z}

\newcommand{\wt}{\widetilde}
\newcommand{\wh}{\widehat}

\DeclareMathOperator{\dom}{Dom} 
\DeclareMathOperator{\re}{Re}

\newtheorem{thm}{Theorem}[section]

\newtheorem{lem}[thm]{Lemma}
\newtheorem{prop}[thm]{Proposition}

\theoremstyle{definition}

\theoremstyle{remark}
\newtheorem{rem}[thm]{Remark}

\numberwithin{equation}{section}

\newcommand{\thmref}[1]{Theorem~\ref{#1}}

\newcommand{\sh}{Schr\"odinger }
\newcommand{\vs}{\vskip0.2cm}

\newcommand{\loc}{\rm{loc}}
\newcommand{\w}{\infty}

\begin{document}

\title[Schr\"odinger operator in $2D$ with radial potential]{On the negative spectrum of the two-dimensional Schr\"odinger operator with radial potential}

\author[A. Laptev]{A. Laptev}
\address{Department of Mathematics\\ Imperial College London\\ Huxley Building\\
180 Queen's Gate\\ London SW7 2AZ, UK}
\email{a.laptev@imperial.ac.uk}
\author[M. Solomyak]{M. Solomyak}
\address{Department of Mathematics
\\ Weizmann Institute\\ Rehovot\\ Israel}
\email{michail.solomyak@weizmann.ac.il}


\begin{abstract}
For a two-dimensional \sh operator $\BH_{\a V}=-\D-\a V$ with the radial potential $V(x)=F(|x|),
F(r)\ge 0$, we study the behavior of the number $N_-(\BH_{\a V})$ of its negative eigenvalues, as the coupling parameter $\a$ tends to infinity. We obtain the necessary and sufficient conditions for the semi-classical growth $N_-(\BH_{\a V})=O(\a)$ and for the validity of the Weyl asymptotic law.
\end{abstract}

\maketitle

\section{Introduction}\label{intro}
Let $\BH_{\a V}$ stand for the \sh operator
\begin{equation}\label{sh}
    \BH_{\a V}=-\D-\a V
\end{equation}
on $\R^d$. We suppose that $V\ge 0$, and $\a>0$ is the coupling constant.

 We write $N(-\g^2,\BH_{\a V})$ for the number of eigenvalues of $\BH_{\a V}$ ({\it bound states}),
lying on the left of the point $\l=-\g^2,\ \g\ge0$, and we denote
\begin{equation}\label{notat}
    N_-(\BH_{\a V})=N(0,\BH_{\a V}).
\end{equation}
If $d\ge3$, the function \eqref{notat} obeys the celebrated Cwikel-Lieb-Rozenblum estimate
\begin{equation}\label{rlc}
    N_-(\BH_{\a V})\le C(d)\a^{d/2}\int_{\R^d}V^{d/2}dx, \ \forall\a>0,\qquad d\ge3,
\end{equation}
under the sole condition that the integral on the right is finite. This estimate is accompanied by
the Weyl-type asymptotic formula
\begin{equation}\label{weyl}
    \lim_{\a\to\infty}\a^{-d/2}N(-\g^2,\BH_{\a V})=(2\pi)^{-d}\om_d\int_{\R^d}V^{d/2}dx,\ \qquad \g\ge0,\
    d\ge 3,
\end{equation}
under the same assumption $V\in L^{d/2}$. In \eqref{weyl} $\om_d$ stands for the volume of the $d$-dimensional unit ball. Comparing the estimate \eqref{rlc} and the asymptotics \eqref{weyl}, we see that the estimate has the correct (semi-classical) order $O(\a^{d/2})$ in the large coupling constant regime.

In the one-dimensional case it is convenient to consider the \sh operator on the half-line $\R_+$:
\begin{equation}\label{1d}
    \BH_{\a G}\varf=-\varf''-\a G\varf,\ \varf(0)=0.
\end{equation}
For this case, an exhaustive description of the potentials guaranteeing the semi-classical
behavior $N_-(\BH_{\a G})=O(\a^{1/2})$ is also known \cite{NS}, though it is expressed in somewhat more complicated terms. A simple sufficient condition for such behavior was obtained as far back as in 1974, see \cite{BSlect}, \S4.8.
The direct analogue of \eqref{rlc} is valid only for the
monotone potentials $V$ (Calogero estimate, see e.g. \cite{ReedSim4}).

The borderline case $d=2$ turns out to be much more difficult, and the exhaustive description of the potentials $V$ guaranteeing the semi-classical behavior
$N_-(\BH_{\a,V})=O(\a)$ is unknown till now. Several attempts to describe such class of potentials, see \cite{BL,LN,S-dim2}, lead only to some, wide enough, sufficient conditions.
Among other papers, devoted to the 2D-case, we would like to mention \cite{MW,St,mv}, though the estimates obtained there have an incorrect order in the large coupling constant regime. Note
that in \cite{mv} the reader finds also a short survey of other relevant results.

The situation changes if we consider the radially symmetric potentials, $V(x)=F(|x|).$
The first result in this direction was obtained in \cite{Lap}. It applies to the operators
\begin{equation*}
    \BH_{b,\a V}=-\D+b|x|^{-2}-\a V,\ \a>0,
\end{equation*}
where $b>0$ is an arbitrary constant, and states that if $d=2$ and $V(x)=F(|x|)\ge 0$, then
\begin{equation}\label{laprad}
    N_-(\BH_{b,\a V})\le A(b)\a \int_0^\infty r F(r)dr,
\end{equation}
with an explicitly given constant factor $A(b)$.
It is well-known that in $\R^2$ any negative potential generates at least one negative bound state, and therefore, a similar estimate for $b=0$ cannot be valid.

Three years later, in the paper \cite{cha}, the following important estimate was obtained for the Hamiltonian \eqref{sh}:
if $d=2$ and $V(x)=F(|x|)$, then
\begin{equation}\label{chad}
    N_-(\BH_{\a V})\le 1+\a\int_0^\infty rF(r)\bigl|\ln\frac{r}{R}\bigr|dr+\frac2{\sqrt3}\a \int_0^\infty rF(r)dr.
\end{equation}
Here $R>0$ can be taken arbitrary, and one can minimize the estimate with respect to $R$. The results in \cite{Lap, cha} are formulated for $\a=1$, but the general case follows immediately by
the substitution $F\mapsto\a F$.
The proofs in both papers make use of the Lieb-Thirring estimate for operators in dimension one.

In a standard way, the estimate \eqref{chad} allows one to justify the asymptotic formula \eqref{weyl} for the radial potentials having the finite integrals involved in \eqref{chad}. For the 2-dimensional case and $\g=0$ this formula takes the form
\begin{equation}\label{as2}
    \lim_{\a\to\infty}\a^{-1}N_-(\BH_{\a V})=\frac1{4\pi}\int_{\R^2} Vdx=\frac12\int_0^\infty rF(r)dr.
\end{equation}
So, the condition that the integrals in \eqref{chad} are finite is sufficient both for the semi-classical growth $N_-(\BH_{\a V})=O(\a)$ and for the validity of the Weyl asymptotics \eqref{as2}. However, it is not necessary for either of these properties.
\vs

The main goal of this paper is to establish the necessary and sufficient conditions in these both problems.
It turns out that they cannot be formulated in terms of $L_1$-spaces with weights, as in \eqref{chad}, but require the so-called "weak" $\ell_1$-space, usually denoted by $\ell_{1,\w}$; see, e.g.,
\cite{BeL}, \S 1.3, for more detail. It is important that the conditions for $N_-(\BH_{\a V})=O(\a)$ and for the validity of \eqref{as2} differ from each other. This is in contrast with the case of $d\ge3$, where the condition $V\in L_{d/2}(\R^d)$ is necessary and sufficient both for the estimate \eqref{rlc} and for the asymptotics \eqref{weyl}. In this respect we would like to notice that the existence of potentials on $\R^2$, for which $N_-(\BH_{\a V})=O(\a)$ but the formula \eqref{as2} fails, had been discovered in \cite{BL}, see Theorem 5.1 there.\vs

Below we remind the definition of the classes $\ell_{1,\w}$.
Consider a sequence of real numbers $\bx=\{x_n\}, \ n=1,2,\ldots$.
One says that this sequence belongs to $\ell_{1,\w}$, if
\begin{equation*}
    \|\bx\|_{1,\w}:=\sup_{\vare>0}\bigl(\vare\#\{n:|x_n|>\vare\}\bigr)
    <\infty.
\end{equation*}
This is a linear space, and the functional $\|\cdot\|_{1,\w}$ defines a quasinorm in it. The latter means that, instead of the standard triangle
inequality, this functional meets a weaker property:
\[  \|\bx+\by\|_{1,\w}\le c\bigl(\|\bx\|_{1,\w}+\|\by\|_{1,\w}\bigr),\]
with some constant $c$ that does not depend on the sequences $\bx,\by$. This quasinorm defines a topology in $\ell_{1,\w}$; there is no norm compatible with this topology. An equivalent, and probably more transparent description of $\|\bx\|_{1,\w}$ can be given by the following property: given a sequence $\{x_n\}$, let $\{x^*_n\}$ stand for the non-increasing rearrangement of the sequence $\{|x_n|\}$. Then
\begin{equation*}
    \|\bx\|_{1,\w}=\sup_n (nx^*_n).
\end{equation*}
So, a sequence $\bx$ belongs to the space $\ell_{1,\w}$
if and only if the non-increasing rearrangement of its absolute values decays as $O(n^{-1})$.

The space $\ell_{1,\w}$ is non-separable. Consider its closed subspace $\ell_{1,\w}^\circ$ in which the sequences $\bx$
with only a finitely many non-zero terms form a dense subset. This subspace is separable, and its elements are
characterized by either of the equivalent conditions
\begin{equation*}
    \bx\in\ell_{1,\w}^\circ \ \Longleftrightarrow\ \vare\#\{n:|x_n|>\vare\}\to0,\ \vare\to0;\qquad x^*_n=o(n^{-1}).
\end{equation*}
It is clear that $\ell_1\subset \ell_{1,\w}^\circ$ and
\begin{equation}\label{compar}
    \|\bx\|_{1,\w}\le\|\bx\|_{\ell_1}=\sum_n|x_n|.
\end{equation}
The (non-linear) functionals
\begin{gather}
    \D_1(\bx)=\limsup_{\vare\to0}(\vare\#\{n:|x_n|>\vare\})=\limsup_{n\to\infty}(nx_n^*),\label{Delta}\\
    \d_1(\bx)=\liminf_{\vare\to0}(\vare\#\{n:|x_n|>\vare\})=\liminf_{n\to\infty}(nx_n^*)\label{delta}
\end{gather}
are well-defined on the space $\ell_{1,\w}$, and clearly,
\[ \d_1(\bx)\le \D_1(\bx)\le \|\bx\|_{1,\w}.\]
It is also clear that $\ell_{1,\w}^\circ=\{\bx\in\ell_{1,\w}:\D_1(\bx)=0\}.$ Note that the values $\D_1(\bx),\d_1(\bx)$ do not change if we remove from $\bx$ any finite number of terms.

To formulate our main results on the estimates for the function $N_-(\BH_{\a V})$ and on the Weyl asymptotics, we need some more notation. Denote $D_0=(0,1)$ and $D_k=(e^{k-1},e^k)$ for $k\in\N$. Given a locally integrable function $G(t)\ge0$ on the half-line $\R_+$, we define the sequence
\begin{equation}\label{def}
    \gz(G)=\{\z_k(G)\}_{k\ge0}:\qquad \z_0(G)=\int_{D_0}G(t)dt,\qquad \z_k(G)=\int_{D_k}tG(t)dt\  (k\in\N).
\end{equation}
If $G$ is defined on the whole line $\R$, we consider
\begin{equation}\label{sum}
   \wh{ \gz}(G)=\{\wh{\z_k}(G)\}_{k\ge0};\quad
     \wh{\z_0}(G)=\int_{-1}^1G(t)dt, \quad
   \wh{\z_k}(G)=\int_{|t|\in D_k}|t|G(t)dt\quad  (k\in\N).
   \end{equation}

\begin{thm}\label{main}
Let $d=2$ and $V(x)=F(|x|)\ge0$. Define an auxiliary one-dimensional potential
\begin{equation}\label{aux}
    G_F(t)=e^{2|t|}F(e^t),\qquad  t\in\R,
\end{equation}
and let $\wh{\gz}(G_F)$ be the corresponding sequence \eqref{sum}. Then
$N_-(\BH_{\a V})=O(\a)$ if and only if $V\in L_1(\R^2)$ and $\wh{\gz}(G_F)\in\ell_{1,\w}$. Under these two assumptions the estimate is satisfied:
\begin{equation}\label{estweak}
    N_-(\BH_{\a V})\le 1+\a\left(\int_0^\infty r F(r)dr +C\|\wh{\gz}(G_F)\|_{1,\w}\right),
\end{equation}
with some constant $C$ independent on $F$.
\end{thm}

\begin{thm}\label{main2}
Let $d=2$, $V(x)=F(|x|)\ge0,\ G_F(t)$ be the potential \eqref{aux}, and let $\wh{\gz}(G_F)$ be the sequence \eqref{sum}. Then the Weyl asymptotic formula \eqref{as2} is satisfied if and only if $V\in L_1(\R^2)$ and $\wh{\gz}(G_F)\in\ell_{1,\w}^\circ$.
\end{thm}

\thmref{main2} is a particular case of a more general result (\thmref{genasy}) that allows also a non-Weyl type behavior of the function $N_-(\BH_{\a V})$. We present it in Subsection \ref{asbeh}.\vs

Now we comment on Theorems \ref{main}, \ref{main2}.

\begin{rem}\label{rem1}
The condition $\wh{\gz}(G_F)\in\ell_{1,\w}$ implies that $\int_0^\infty r F(r)dr<\infty$ and thus $V\in L_1(\R^2)$. Indeed,
for $k\in\N$ we have
\[ \wh{\z_k}(G_F)\ge e^{k-1}\int_{|t|\in D_k} G_F(t)dt.\]
Any sequence from $\ell_{1,\w}$ is bounded, and hence,
\[ \int_{|t|\in D_k}G_F(t)dt\le C e^{-(k-1)}.\]
It follows that the terms on the left form a convergent series. Its sum is equal to
\[ \int_\R e^{2t}F(e^t)dt=\int_0^\infty r F(r)dr,\]
and we are done.

So, for the radial potentials $V$ the inclusion $V\in L_1(\R^2)$ follows from the assumption $\wh{\gz}(G_F)\in\ell_{1,\w}$ and hence, could be omitted from the formulation of both theorems. Respectively, the estimate \eqref{estweak} could
be replaced by
\[ N_-(\BH_{\a V}) \le 1+C\a\|\wh{\gz}(G_F)\|_{1,\infty}.\]
However, we still consider it useful to mention the condition $V\in L_1(\R^2)$ explicitly.
\end{rem}

\begin{rem}\label{rem2}
Denote $G_F^\pm(t)=G_F(\pm t),\ t>0$, then $\wh{\gz}(G_F)=\gz(G_F^+)+\gz(G_F^-)$. The sequence
$\gz(G_F^+)$ controls the behavior of $F(r)$ as $r\to\infty$, and $\gz(G_F^-)$ does this as $r\to 0$. In this respect we would like to note that the auxiliary one-dimensional potential $G_F^+$ was introduced in the paper \cite{S-dim2}, where arbitrary (not necessarily radial) potentials $V$ were considered. The function $F$ was defined as
\[ F(r)=\frac1{2\pi}\int_0^{2\pi} V(r,\vart)d\vart,\]
where $r,\vart$ are the polar coordinates on $\R^2$. In \cite{S-dim2} some estimates for $N_-(\BH_{\a V})$ where obtained. They included the term $\|\gz(G_F^+)\|_{1,\w}$ as an important ingredient in the right-hand side. Independently, the same potential $G_F^+$ appeared in the paper \cite{BL}, where its role in the asymptotic formulas for $N_-(\BH_{\a V})$ was elucidated.
The potential $G_F^-$ did not appear in the formulations in \cite{S-dim2, BL}. This happened,
since other assumptions about $V$, stronger than just $V\in L_1$, already implied the necessary behavior of $F(r)$ as $r\to 0$. For instance, in \cite{BL} it was assumed that $V\in L_{\s,\loc}(\R^2)$ with some $\s>1$. Now, for the radial potentials, the assumption $V\in L_1$ is much weaker, and this makes it necessary to have an additional restriction on
the behavior of $F(r),\ r\to 0$.
 \end{rem}

 In the next three remarks we compare our results with those of \cite{cha}, and give some examples.
\begin{rem}\label{rem3}
By the definition \eqref{sum} and the inequality \eqref{compar}, we have
\begin{gather*}
\|\wh{\gz}(G_F)\|_{1,\w}\le \|\wh{\gz}(G_F)\|_{\ell_1}
=\int_{-1}^1 e^{2t}F(e^t)dt+\int_{|t|>1} |t|e^{2t} F(e^t)dt\\ =\int_0^\infty r F(r)\left(1+(|\ln r|-1)_+\right)dr.
\end{gather*}
It immediately follows that the estimate \eqref{estweak} is stronger than \eqref{chad}.\vs

Still, the bound \eqref{chad} has an evident merit, since it estimates the number $N_-(\BH_{\a V})$ in terms of some explicitly given integrals. In contrast to this, the derivation of \eqref{estweak} makes use of the real interpolation method, which does not allow to specify the optimal value of $C$ in this estimate.

In this respect we note that the factor $2/\sqrt3$ in \eqref{chad} can be removed. We show this
in Subsection \ref{sharp}. The possibility of this removal was formulated in \cite{cha} as a conjecture.
\end{rem}
\begin{rem}\label{rem4}
Consider the function
\[F(r)=\begin{cases} 0,& r\le e^{e^2};\\
r^{-2}(\ln r)^{-2}(\ln\ln r)^{-1},& r>e^{e^2}.
\end{cases}\]
Then, for $t>2$, we have $G_F(t)=t^{-2}(\ln t)^{-1}$, and hence, $\wh{\z}_k(G_F)=\ln{\frac{k}{k-1}}\asymp k^{-1}$ for large $k$. According to Theorems \ref{main},
\ref{main2},
for the potential $V(x)=F(|x|)$ we have $N_-(\BH_{\a V})=O(\a)$, but the asymptotic formula \eqref{as2} fails. The estimate \eqref{chad} does not apply to this example.\end{rem}
\begin{rem}\label{rem5}
In the previous example, let us multiply $F(r)$ by a bounded function $\psi(r)$ vanishing as
$r\to \infty$. For the resulting function $F\psi$ the corresponding sequence ${\wh{\gz}(G_{F\psi})}$ belongs to the space $\ell_{1,\w}^\circ$ and therefore, for the potential $V(x)=F(|x|)\psi(|x|)$
the estimate \eqref{estweak} and the asymptotic formula \eqref{as2} are satisfied. If $\psi(r)$ decays slowly enough, then
$\int_0^\infty rF(r)\psi(r)|\ln r|dr=\infty$, and again,
the estimate \eqref{chad} does not apply to the potential $F\psi$.
\end{rem}\vs

\medskip
\noindent
{\it Acknowledgements.}
The authors express their gratitude to G. Rozenblum and to the referees for valuable remarks .

\section{Auxiliary material}\label{auxres}
The proofs of our basic results mainly follow the line worked out in \cite{BL, S-dim2} and \cite{Lap}. Still, there are some distinctions, and we prefer to give an independent exposition. We systematically use the variational description of spectra.
\subsection{Classes $\Sg_1,\Sg_1^\circ$ of compact operators}\label{clas}
Along with the spaces $\ell_{1,\w},\ \ell_{1,\w}^\circ$ of number sequences, we need
the corresponding spaces of compact operators in the Hilbert space. If $\BT$ is such operator,
then, as usual, $\{s_n(\BT)\}$ stands for the sequence of its singular numbers, i.e., for the
eigenvalues of the non-negative, self-adjoint operator $(\BT^*\BT)^{1/2}$. We say that $\BT$ belongs to the class $\Sg_1$ if and only if $\{s_n(\BT)\}\in\ell_{1,\w}$, and to the class $\Sg_1^\circ$ if and only if $\{s_n(\BT)\}\in\ell_{1,\w}^\circ$. These are linear, quasinormed
spaces with respect to the quasinorm $\|\BT\|_{1,\w}$ induced by this definition. Evidently, $\Sg_1^\circ\subset\Sg_1$. The space $\Sg_1$ is non-separable, and $\Sg_1^\circ$ is its
separable subspace in which the finite rank operators form a dense subset. It is clear that
the trace class $\GS_1$ is contained in $\Sg_1^\circ$, and $\|\BT\|_{1,w}\le \|\BT\|_{\GS_1}$. Similarly to \eqref{Delta}, \eqref{delta}, we define
the functionals
\begin{equation}\label{del-op}
    \D_1(\BT)=\D_1(\{s_n(\BT)\}),\qquad \d_1(\BT)=\d_1(\{s_n(\BT)\}).
\end{equation}
The values of these functionals do not change if we add to $\BT$ any finite rank operator. Note also that
\begin{equation}\label{del-ineq}
    \d_1(\BT)\le \D_1(\BT)\le\|\BT\|_{1,\w}.
\end{equation}

In fact,  the spaces $\Sg_q,\Sg_q^\circ$ of compact operators are well-defined for all $q>0$,
see \cite{BSbook}, Section 11.6. However, in this paper we are dealing with $q=1$ only.

\subsection{Birman-Schwinger principle.}\label{B-Shw} As it is standard in this type of problems, our approach is based upon the classical Birman-Schwinger principle. In its general form, it was stated by Birman \cite{B1}.

Let $Q[u]$ be a densely defined, positive and closed quadratic form in a Hilbert space $\GH$. Suppose that a real-valued quadratic form $\bb[u]$ is non-negative and $Q$-bounded, that is,
\begin{equation}\label{1}
    \bb[u]\le CQ[u],\qquad u\in\dom Q.
\end{equation}
Along with $\GH$, consider another Hilbert space $\GH_Q$, namely the completion of $\dom Q$ with respect to the norm $\sqrt{Q[u]}$. Due to \eqref{1}, the quadratic form $\bb[u]$ extends by continuity to the whole of $\GH_Q$. Denote by $\BT_\bb$ the bounded operator in
$\GH_Q$, generated by this extended quadratic form. Consider also the family of quadratic forms in the original Hilbert space $\GH$, depending on the parameter $\a>0$:
\begin{equation}\label{2}
    Q_{\a\bb}[u]=Q[u]-\a b[u],\qquad \dom Q_{\a\bb}=\dom Q.
\end{equation}
\begin{prop}\label{birschw} $($Birman-Schwinger principle$)$. Under the above assumptions, suppose in addition that the operator $\BT_\bb$
is compact. Then for any $\a>0$ the quadratic form \eqref{2} is bounded from below and closed, the negative spectrum of the corresponding operator $\BA_\bb(\a)$ is finite, and the equality for the number of its negative eigenvalues is satisfied:
\begin{equation}\label{3}
    N_-(\BA_\bb(\a))=n_+(\a^{-1},\BT_\bb)=\#\{n:\l_n(\BT_\bb)
    >\a^{-1}\},\qquad\forall\a>0.
\end{equation}

Conversely, if the quadratic form $Q_{\a\bb}[u]$ is bounded from below and closed  and
the number $N_-(\BA_\bb(\a))$ is finite for all $\a>0$, then the operator $\BT_\bb$ is compact and \eqref{3} is satisfied.
\end{prop}
It follows directly from Proposition \ref{birschw} that the estimate
\begin{equation}\label{4}
    N_-(\BA_\bb(\a))\le C\a
\end{equation}
is equivalent to the inclusion $\BT_\bb\in\Sg_1$, and the sharp value of the constant $C$ in
\eqref{4} coincides with $\|\BT_\bb\|_{1,\w}$. Also, the asymptotic characteristics of
the function $N_-(\BA_\bb(\a))$ can be expressed in terms of the operator $\BT_\bb$:
\begin{gather*}
    \limsup_{\a\to\infty}\a^{-1}N_-(\BA_\bb(\a))=\D_1(\BT_\bb);\qquad
    \liminf_{\a\to\infty}\a^{-1}N_-(\BA_\bb(\a))=\d_1(\BT_\bb).
\end{gather*}\vs

The operator $\BT_\bb$ is usually called the Birman-Schwinger operator for the family of the quadratic forms $Q_{\a\bb}$, see \eqref{2}, or for the family of the corresponding operators $\BA_\bb(\a)$ in $\GH$.

\subsection{Reduction of the main problem to compact operators}\label{comp}
In the space $C_0^\infty=C_0^\infty(\R^2)$ let us introduce two subspaces, 
\begin{equation*}
    \CF_0=\{f\in C_0^\infty: f(x)=\varf(r), \varf(1)=0\};\qquad \CF_1=\{f\in C_0^\infty: \int_0^{2\pi}f(r,\vart)d\vart=0,\ \forall r>0\}.
\end{equation*}
Here $r,\vart$ stand for the polar coordinates on $\R^2$.

These subspaces are orthogonal both in the $L_2$-metric and in the metric of the Dirichlet integral. The Hardy inequalities have a different form on $\CF_0$ and on $\CF_1$:
\begin{multline}\label{H0}
    \int\frac{|f|^2}{|x|^2\ln^2|x|}dx=2\pi\int_0^\infty\frac{|\varf(r)|^2}{r\ln^2 r}dr\\
    \le\frac{\pi}2\int_0^\infty r|\varf'(r)|^2dr=\frac14\int|\nabla f|^2dx,\qquad f\in \CF_0;
\end{multline}
\begin{equation}\label{H1}
    \int\frac{|f|^2}{|x|^2}dx\le\int|\nabla f|^2dx,\qquad f\in \CF_1.
\end{equation}
Here and later on, the integral with no domain specified always means $\int_{\R^2}$.

For proving \eqref{H0}, one substitutes $r=e^t$, and then applies the standard Hardy inequality in dimension 1, i.e.,
\[\int_0^\infty\frac{|\psi(t)|^2}{t^2}dt< 4\int_0^\infty |\psi'(t)|^2dt\]
that is satisfied for any absolute continuous function $\psi\not\equiv0$ on $\R_+$, such that $\psi'\in L_2$ and $\psi(0)=0$;
see, e.g., \cite{hlp}.

The proof of \eqref{H1} is quite elementary, it can be found, e.g., in \cite{S-dim2},
or in \cite{BL}.

\begin{rem}\label{R}
Without the additional condition $\varf(1)=0$ the inequality \eqref{H0} fails. Instead of
this condition, it is possible to require $\varf(R)=0$
at any chosen point $R>0$. Then the term $\ln^2|x|$ in the first integral in \eqref{H0} should be
replaced by $\ln^2(|x|/R)$.
\end{rem}
\vs

 The difference between the inequalities \eqref{H1} and \eqref{H0} is reflected in the estimate \eqref{estweak}: as we shall see, the first term in brackets on the right-hand side of \eqref{estweak} is responsible for the estimates on the subspace $\CF_1$, while the second term is responsible for those on $\CF_0$.\vs

 Let us consider the completions $\CH^1_0, \CH^1_1$ of the spaces $\CF_0,\CF_1$ in the metric of the Dirichlet integral. It follows from the Hardy inequalities \eqref{H0}, \eqref{H1} that these are Hilbert function spaces, embedded into the weighted $L_2$, with the weights defined by these inequalities. Consider also their orthogonal sum
  \begin{equation}\label{ortsum}
    \CH^1=\CH_0^1\oplus\CH_1^1.
 \end{equation}
 An independent definition of this Hilbert space is
 \begin{equation*}
    \CH^1=\{f\in H^1_{\loc}(\R^2): \int_0^{2\pi}f(1,\vart)d\vart=0,\ |\nabla f|\in L_2(\R^2)\},
 \end{equation*}
 with the metric of the Dirichlet integral.

 Later we will need also the spaces $\CG_0, \CG_1$ which are the completions of $\CF_0,\CF_1$ in
 the $L_2$-metric. Note that the condition $\varf(1)=0$, occuring in the description of $\CF_0$, disappears for general $f\in\CG_0$.

 Suppose that $V\ge0$ is a measurable function, such that
 \begin{equation}\label{qform}
    \bb_V[u]:=\int V|u|^2dx\le C\int|\nabla u|^2dx,\qquad\forall u\in\CH^1.
 \end{equation}
 Under the assumption \eqref{qform} the quadratic form $\bb_V$ defines a bounded
 self-adjoint operator $\BB_V\ge0$ in $\CH^1$. If (and only if) this operator is compact, then, by the Birman-Schwinger principle, the quadratic form
 \begin{equation*}
    \int(|\nabla u|^2-\a V|u|^2)dx
 \end{equation*}
 with the form-domain $\{u\in H^1(\R^2): \int_0^{2\pi}u(1,\vart)d\vart=0\}$
 is closed and bounded from below for each $\a>0$, the negative spectrum of the associated self-adjoint operator $\wt{\BH}_{\a V}$ on $L_2(\R^2)$ is finite, and the following equality for the number of its negative eigenvalues holds true:
 \begin{equation}\label{bshw}
    N_-(\wt{\BH}_{\a V})=n_+(\a^{-1},\BB_V),\qquad\forall\a>0.
 \end{equation}

 Now, let us withdraw the rank one condition $\int_0^{2\pi}u(1,\vart)d\vart=0$ from the description of the form-domain. Then the resulting quadratic form corresponds to the \sh operator $\BH_{\a V}$. Hence,
 \[ N_-(\wt{\BH}_{\a V})\le N_-(\BH_{\a V})\le N_-(\wt{\BH}_{\a V})+1,\]
 and, by \eqref{bshw},
 \begin{equation}\label{equal}
    n_+(\a^{-1},\BB_V)\le N_-(\BH_{\a V})\le n_+(\a^{-1},\BB_V)+1.
 \end{equation}
 Thus, the study of the quantity $N_-(\BH_{\a V})$ for all $\a>0$ is reduced to the investigation of the "individual" operator $\BB_V$.
 \subsection{Orthogonal decomposition of the operator $\BB_V$.}\label{scpr}
 Given a function $u\in\CH^1$, we agree to standardly denote its components in the decomposition \eqref{ortsum} by $\varf(r),v(r,\vart)$. Along with the quadratic form $\bb_V$, we consider its "parts" in the subspaces $\CH_0^1,\CH_1^1$:
 \begin{equation}\label{parts}
    \bb_{V,0}[u]=\bb_V[\varf],\qquad \bb_{V,1}[u]=\bb_V[v].
 \end{equation}
 Let $\BB_{V,j},\ j=0,1,$ stand for the corresponding self-adjoint operators in $\CH_j^1$. Using the orthogonal decomposition \eqref{ortsum}, we see that
 \begin{equation}\label{dec}
    \bb_V[u]=\bb_{V,0}[\varf]+\bb_{V,1}[v]+2\int F(|x|)\re(\varf(|x|)\overline{v(x)})dx.
\end{equation}
 Here the last term vanishes, $\int F(|x|)\re(\varf(|x|)\overline{v(x)})dx=0$, due to the orthogonality of $v$ (in $L_2$) to all functions depending
 on $|x|$. So, the decomposition \eqref{ortsum} diagonalizes the quadratic form $\bb_V$:
 \begin{equation}\label{dec1}
    \bb_V[u]=\bb_{V,0}[\varf]+\bb_{V,1}[v].
\end{equation}
 Equivalently, we have
 \begin{equation}\label{dec2}
    \BB_V=\BB_{V,0}\oplus\BB_{V,1},
 \end{equation}
 whence
 \begin{equation}\label{countfunc}
    n_+(s,\BB_V)= n_+(s,\BB_{V,0})+n_+(s,\BB_{V,1}),\qquad \forall s>0.
 \end{equation}
 This immediately shows that the proof of both theorems \ref{main} and \ref{main2} is reduced to the independent study of  the two operators $\BB_{V,j},\ j=0,1$.
It also shows that for the radial potentials the subspaces $\CG_0,\CG_1$ reduce the Hamiltonian
\eqref{sh} to the diagonal form.

\section{Operator $\BB_{V,0}$. }\label{dim1}
The Rayleigh quotient that corresponds to the operator $\BB_{V,0}$ is
\[\frac{\int_0^\infty r F(r)|\varf(r)|^2dr}{\int_0^\infty r|\varf'(r)|^2dr},\qquad \varf(1)=0.\]
The standard substitution $r=e^t,\ \varf(r)=\om(t);\ t\in\R$, reduces it to the form
\begin{equation}\label{ray-ax}
    \frac{\int_\R G_F(t)|\om(t)|^2dt}{\int_\R|\om'(t)|^2dt},\qquad \om(0)=0.
\end{equation}
where $G_F(t)=e^{2t}F(e^t)$.
Due to the boundary condition at the point $t=0$, the corresponding operator decomposes into the direct orthogonal sum of two operators, each acting on a half-line. In particular, the one for the half-line $\R_+$ is
nothing but the Birman-Schwinger operator for the family \eqref{1d}, with $G=G_F$. The spectral estimates for such operators were studied in detail in \cite{BL}, Section 4. Below we present only those results of this study that we need in this paper.
\begin{prop}\label{diffop}
Let $G(t)\ge 0$ be a function on $\R_+$, integrable on
each finite interval $(0,A)$, and let $\BT_G$ be the Birman-Schwinger operator that corresponds to the family \eqref{1d}. Assume that $\gz(G)$ is the sequence defined in \eqref{def}.
The condition $\gz(G)\in\ell_{1,\w}$ is necessary and sufficient for the inclusion
$\BT_G\in\Sg_1$. Moreover, the estimates are satisfied:
\begin{gather}
    \|\BT_G\|_{1,\w}\le C\|\gz(G)\|_{1,\w}\label{est1};\\
    c\D_1(\gz(G))\le \D_1(\BT_G)\le C\D_1(\gz(G))\label{est2}.
\end{gather}
\end{prop}

We remind that the functional $\D_1$ for a number sequence was defined in \eqref{Delta},
and for a compact operator in \eqref{del-op}. Note that the behavior of the eigenvalues of the operator $\BT_G$ under the assumptions of Proposition \ref{diffop} may be quite irregular, and only a few examples are known where one has $\l_n(\BT_G)\sim c_0 n^{-1}$ with some $c_0>0$. This is the reason why we describe this behavior in terms of the functionals $\D_1,\d_1$. Note also that, unlike \eqref{est2}, the estimate \eqref{est1} cannot be inverted, since the term $\z_0(G)$ is not controlled by $\|\gz(G)\|_{1,\w}$.

It is clear that the result extends to the case of the whole
line (but under the condition $\om(0)=0$). The only difference is that instead of the sequence
$\gz(G)$ one should consider the sequence $\wh{\gz}(G)$ defined by \eqref{sum}.

So, we arrive at the following
\begin{lem}\label{B0}
Let a function $F(r)\ge0$ be integrable on each finite interval $(a,A)\subset(0,\infty)$, and let $G_F(t)=e^{2t}F(e^t),\ t\in\R$. The operator $\BB_{V,0}$ belongs to the class $\Sg_1$ if and only if the sequence $\wh{\gz}(G_V)$ defined by \eqref{sum} lies in the space $\ell_{1,\w}$, and the following estimates are satisfied:
\begin{gather*}
    \|\BB_{V,0}\|_{1,\w}\le C\|\wh{\gz}(G_F)\|_{1,\w};\label{estB01}\\
    c\D_1(\wh{\gz}(G_F))\le \D_1(\BB_{V,0})\le C\D_1(\wh{\gz}(G_F)).
\end{gather*}
\end{lem}

\section{Operator $\BB_{V,1}$. Proof of \thmref{main}.}\label{prmain}
\subsection{Operator $\BB_{V,1}$. }\label{B1}
Our goal here is to establish the following result.
\begin{lem}\label{bv1}
Let $F(r)\ge0$. Then $\BB_{V,1}\in\Sg_1$ if and only if
$J(F):=\int_0^\infty rF(r)dr<\infty$. Under this condition we have
\begin{equation}\label{b1est}
    \|\BB_{V,1}\|_{1,\w}\le J(F)
\end{equation}
and
\begin{equation*}
    \D_1(\BB_{V,1})=\d_1(\BB_{V,1})=\frac{J(F)}2,
\end{equation*}
or, in other words, the following limit exists,
\begin{equation}\label{b1as}
    \lim_{s\to 0} sn_+(s,\BB_{V,1})=\frac12\int_0^\infty rF(r)dr.
\end{equation}
\end{lem}
\begin{proof}
What is given below, is a slight modification of the original argument in \cite{Lap}. A direct use
of the estimate \eqref{laprad} would give a result similar to
\eqref{b1est}, but with an excessive numerical factor.

Let $\BH_{\a V,1}$ stand for the restriction of the Hamiltonian \eqref{sh} to the
subspace $\CG_1$. Since the potential is radial, this subspace is invariant for $\BH_{\a V,1}$. Note that the operator $\BB_{V,1}$ is just the Birman-Schwinger operator for the family $\BH_{\a V,1}$.

It is enough to deal with the value $\a=1$, and
it is more convenient here to estimate the number $N_-(\BH_{V,1})$ directly, rather than
via evaluating the quasinorm $\|\BB_{V,1}\|_{1,\infty}$. Let us consider the
corresponding quadratic form which is
\begin{equation}\label{hv}
    \bh_{V,1}[v]=\int(|\nabla v|^2-F(|x|)|v|^2)dx=\int_0^\infty\int_0^{2\pi}\left(|v'_r|^2+
    r^{-2}|v'_\vart|^2-F(r)|v|^2\right)rdrd\vart.
    \end{equation}
By setting $r=e^t,\ w(t,\vart)=v(e^t,\vart)$, we reduce it to
\begin{equation}\label{hv1}
    \wt{\bh}_{V,1}[w]=
    \int_\R\int_0^{2\pi}\left(|w'_t|^2+|w'_\vart|^2-G_F(t)|w|^2\right)dtd\vart.
\end{equation}
Let $-\mu_k,\ \mu_k>0,$ stand for the eigenvalues of the \sh operator $-\psi''(t)-G_F(t)\psi(t)$
on the line.
Separation of variables shows that the negative spectrum of the operator generated by
the quadratic form \eqref{hv1}, and hence the one of
$\BH_{V,1}$, consists of the eigenvalues $-(\mu_k-n^2)$, where $n\in \Z\setminus\{0\},\ k\in\N$,
and $n^2<\mu_k$. For each $k$, the number of all admissible values of $n$ does not exceed $2\mu_k^{1/2}$, and therefore
\[ N_-(\BH_{V,1})\le 2\sum_{k\in\N}\mu_k^{1/2}.\]
Now, using the Lieb-Thirring inequality in the dimension $1$ for the borderline value of the exponent (which is $1/2$), with the sharp constant $1/2$, see \cite{hun}, we
conclude that
\begin{equation}\label{est-hun}
    N_-(\BH_{V,1})\le \int_\R G_F(t)dt=J(F).
\end{equation}
This is equivalent to \eqref{b1est}.\vs

To justify the asymptotic formula \eqref{b1as}, we consider first the case $F\in C_0^\infty(0,\infty)$. Then Theorem 5.1 in \cite{BL} applies and it shows that
\[ 2\lim_{\a\to 0}\a^{-1} N_+(\BH_{\a V})=J(F).\]
By the Birman-Schwinger principle, this is equivalent to
\begin{equation}\label{asB1}
    2sn_+(s,\BB_V)\to J(F),\qquad s\to 0.
\end{equation}
The spectrum of $\BB_{V,1}$ has the same asymptotic behavior, since for such potentials the subspace $\CH^1_0$ does not contribute to the asymptotic coefficient.

Now, let $F\ge0$ be an arbitrary function, such that $J(F)<\infty$. Then, approximating it by functions $F\in C_0^\infty(0,\infty)$ and taking into account the continuity of the asymptotic coefficients in the metric of $\Sg_1$ (see \cite{BSbook}, Theorem 11.6.6), we extend the formula \eqref{asB1} to any $F$ under consideration. This concludes the proof of the "if" part of Lemma.

Suppose now that $F\ge0$ and $\BB_V\in\Sg_1$. Take any bounded and compactly supported function $\wt F$, such that $0\le\wt F\le F$, then for the potential $\wt V(x)=\wt F(|x|)$ we
obtain from the variational principle:
\[ 2\lim_{s\to0}sn_+(s,\BB_{\wt V})=\int_0^\infty r\wt F(r)dr\le 2\|\BB_V\|_{1,\w}.\]
By Fatou's lemma and the inequality \eqref{del-ineq}, this yields $\int_0^\infty r F(r)dr\le  \|\BB_V\|_{1,\w}$.
\end{proof}

\subsection{Proof of \thmref{main}.}\label{pr}
We are now able to prove \thmref{main}. To establish \eqref{estweak}, we use \eqref{equal} and \eqref{countfunc}.  Then we estimate  $n_+(s,\BB_{V,0})$ using Lemma \ref{B0}, and  we estimate $n_+(s,\BB_{V,1})$ using Lemma \ref{bv1}. This leads directly to \eqref{estweak}, and gives the proof of the sufficiency part of Theorem.

The necessity part follows immediately from the evident inequality

\[ \max\left(\|\BB_{V,0}\|_{1,\w}, \|\BB_{V,1}\|_{1,\w}\right)\le  \|\BB_V\|_{1,\w}.\]

\section{Some complementary results}\label{compl}
Here we state and prove the general result on the asymptotic behavior of $N_-(\BH_{\a V})$,
promised in Introduction, and show that the constant factor $2/\sqrt3$ in the estimate \eqref{chad} can be removed.
\subsection{ Asymptotic behavior of $N_-(\BH_{\a V})$ in the general case.}\label{asbeh}
Here we prove
\begin{thm}\label{genasy}
Under the assumptions of \thmref{main} one has
\begin{gather*}
\limsup_{\a\to\infty}\a^{-1}N_-(\BH_{\a V})=\frac12\int_0^\infty r F(r)dr+\D_1(\BT_{G_F}),\\
\liminf_{\a\to\infty}\a^{-1}N_-(\BH_{\a V})=\frac12\int_0^\infty r F(r)dr+\d_1(\BT_{G_F}).
\end{gather*}
\end{thm}

It is clear that \thmref{main2} is just a special case of \thmref{genasy}, when $\D_1(\wh{\gz}(G_F))=0$ and hence, $\D_1(\BT_{G_F})=\d_1(\BT_{G_F})=0$.

\begin{proof} The result immediately follows from
the equality \eqref{dec2}, in view of Lemmas \ref{B0}
and \ref{bv1}.
\end{proof}

\subsection{Sharpening of the estimate \eqref{chad}}\label{sharp}
We restrict ourselves to the case $R=1$. To handle the case of general $R$, it is sufficient
to take into account Remark \ref{R} in Subsection \ref{comp}.
\begin{thm} Let both integrals in \eqref{chad} be finite. Then
\begin{equation}\label{chad1}
    N_-(\BH_{\a V})\le 1+\a\int_0^\infty rF(r)\bigl|\ln r\bigr|dr+\a \int_0^\infty rF(r)dr.
\end{equation}
\end{thm}
\begin{proof}
By Lemma \ref{bv1}, we have
\[n_+(s,\BB_{V,1})\le J(F)s^{-1}.\]
For $\BB_{V,0}$ we apply the classical Bargmann estimate, (see, e.g., \cite{ReedSim4} to each of the operators  corresponding to the Rayleigh quotient \eqref{ray-ax}. This gives
\[ n_+(s,\BB_{V,1})\le \int_\R |t|G_F(t)dt\, s^{-1}=\int_0^\infty rF(r)|\ln r|dr\, S^{-1}.\]
Now we conclude from \eqref{dec2} that
\[ n_+(s,\BB_V)\le(J(F)+ \int_\R |t|G_F(t)dt) s^{-1}.\]
Taking  \eqref{equal} into account, we arrive at the desired estimate \eqref{chad1}.
\end{proof}


\end{document}